\documentclass[12pt,reqno]{amsart}
\usepackage{amsthm,amsfonts,amssymb,euscript}

\def\S{\mathbb{S}^{n-1}}
\def\a{{\alpha}}
\def\b{{\beta}}
\def\de{\delta}

\def\G{\Gamma}
\def\la{\lambda}
\def\p{{\partial}}
\def\rh{{\hat r}}
\def\Rh{{\hat R}}
\def\g{\gamma}
\def\nab{\nabla}

\def\gb{\bar{g}}
\def\ub{\bar{u}}
\def\vb{\bar{v}}
\def\Rb{\bar{R}}
\def\phib{\bar{\phi}}
\newtheorem{theorem}{Theorem}[section]
\newtheorem{lemma}[theorem]{Lemma}
\newtheorem{proposition}[theorem]{Proposition}
\newtheorem{corollary}[theorem]{Corollary}

\newtheorem{remark}[theorem]{Remark}

\numberwithin{equation}{section}

\begin{document}\title[On Strong Unique Continuation]{On Strong
Unique Continuation of Coupled Einstein Metrics}
\author{Willie Wai-Yeung Wong}
\address{Department of Pure Mathematics and Mathematical Statistics\\ University of Cambridge}
\email{ww278@dpmms.cam.ac.uk}
\thanks{Most of this work was completed when W.W.Y.W.~was a graduate
student at Princeton University.}
\author{Pin Yu}
\address{Department of Mathematics\\ Princeton University}
\email{pinyu@math.princeton.edu}

\begin{abstract}
The strong unique continuation property for Einstein metrics can be
concluded from the well-known fact that Einstein metrics are analytic
in geodesic normal co\"ordinates. Here we give a different proof of 
that result using a Carleman inequality (and thus circumventing
the use of analyticity). Hence the method is robust under certain
non-analytic perturbations. As an example, we also show the strong
unique continuation property for the Riemannian Einstein-scalar-field
system with cosmological constant.

\end{abstract}
\maketitle

\section{Introduction}

An Einstein metric $g$ on a Riemannian manifold $M$ satisfies the
following system of equations
\begin{equation}\label{Einstein}
R_{ij}=\lambda g_{ij}
\end{equation}
where $R_{ij}$ is the Ricci curvature of the given metric $g_{ij}$
and the Einstein constant $\lambda$ is a fixed real number. The system
is essentially elliptic (a fact made most abundantly clear in harmonic
co\"ordinates), and in fact a solution can be seen as real-analytic in
harmonic or geodesic normal co\"ordinates \cite{Besse, DK}. It is
simple to see that then the strong unique continuation 
property\footnote{The strong unique continuation property for a system 
of equations can be stated roughly as: if $u$ and $v$ are solutions and
$p$ a point such that $u(p) = v(p)$ and $\partial^\alpha u(p) =
\partial^\alpha v(p)$ for any multi-index $\alpha$ (in other words,
their Taylor expansions agree at the point $p$), then $u = v$
everywhere.  See Theorem \ref{maintheorem} for precise conditions used
in this paper.} must hold for Einstein manifolds. 

Here we offer a different proof of the strong unique continuation
property using Carleman estimates. The Carleman inequality has a long
history in connection with uniqueness theorems (see \cite{Sogge} for a
selective history; also see \cite{Biquard, Kazdan, Taubes} for
examples) for elliptic systems. Recently it has also found use for
uniqueness properties of an ill-posed characteristic hyperbolic system
\cite{Alexakis, AIK}. The main advantage of an approach based on
Carleman estimates is that it relaxes the analyticity requirement.
Indeed, our proof is based on an estimate of Sogge's \cite{Sogge} that
allows even rough coefficients in the equation. To illustrate the
power of this method, we also prove the strong unique continuation
property for the Einstein equation coupled to a scalar field with an
potential that is $C^{3,1}$. 

To apply Carleman estimates, we interpret the Einstein(-scalar-field)
system as an elliptic system for the Riemann curvature tensor. By
the second Bianchi identities, we have that schematically
\[ dR = 0~;\quad \mathop{div} R = dRic~,\]
where the derivative of the Ricci tensor is zero in the pure Einstein
case. From the divergence-curl system we can derive the schematic equation
\[ \triangle_g R = R*R + \nabla^2_gRic~,\]
where $R*R$ denotes some fixed tensorial contractions between the Riemann
curvature tensor and itself, and especially is quadratic in $R$; and
$\nabla_g$ is the Levi-Civita connection of $g$.
The main idea is then to show that this system has unique continuation
properties: in other words, let $g$ and $\bar{g}$ be two metrics
solving our system, we wish to show that their corresponding Riemann
tensors $R, \bar{R}$ agree in a neighborhood. Subtracting the
\emph{equations}, we discover that, schematically, we need to consider
\[ \triangle_g (R - \bar{R}) =
-(\triangle_g-\triangle_{\bar{g}})\bar{R} + (R-\bar{R})*R +
\bar{R}*(R-\bar{R}) + \nabla_g^2(Ric - \bar{Ric}) +
(\nabla_g^2-\nabla_{\bar{g}}^2)\bar{Ric}~.\]
Here we immediately see two obstructions to directly applying the
Carleman estimates to obtain strong unique continuation, both of which
have to do with the second derivative terms. 

On the one hand, there is
the term $\nabla_g^2(Ric - \bar{Ric})$, which is the same order of
derivatives as the left hand side of the equation. A principle behind
the use of Carleman estimates is that to obtain unique continuation the
equations must be \emph{order reducing} (the Carleman inequality
itself estimates lower order derivatives by the top order derivatives
with a small constant. An order reducing equation estimates the top
order derivative by lower order derivatives with constant 1. Combining
the two gives uniqueness). For the pure Einstein case, the Ricci
tensor has zero covariant derivative, so this obstruction is
side-stepped. In the model case with the scalar field (this is also
true for more general matter fields such as Maxwell or Yang-Mills),
the derivatives of Ricci can be estimated by the derivatives of the
scalar field. As long as the matter field equations are minimally coupled
(which means that the geometry only appears in the form of covariant
differentiation and metric contraction), we can hope that prolongation
of the field equations leads to estimable quantities for the higher
derivatives. To be more precise: for the Einstein-scalar equations, we
have
\[ Ric = \partial\phi \otimes\partial\phi + V(\phi) g \]
where $\phi$ satisfies a scalar equation, and 
\[ \triangle_g\phi = V'(\phi) \]
where the potential $V$ satisfies some regularity conditions.
Therefore to control $\nabla_g^2(Ric - \bar{Ric})$, the worst term
comes from $\nabla_g^3(\phi - \bar{\phi})$, which we can obtain by
prolonging the scalar equation with two derivatives, whilst
sacrificing one derivative on Riemann curvature:
\[ \triangle_g \nabla^2_g\phi = \nabla^2_gV'(\phi) +
\nabla(R\cdot\nabla\phi) ~.\]
So in the end we obtain a system of equations for $\triangle_g\phi,
\triangle_g\nabla_g\phi, \triangle_g\nabla^2_g\phi$, and
$\triangle_gR$, which is order-reducing when all terms are considered.
It is for this system that we consider uniqueness. 

The second obstruction arises from the difference in connections,
which appears as $\nabla_g-\nabla_{\bar{g}}$ or $\nabla^2_g -
\nabla^2_{\bar{g}}$ in the subtracted equations. This obstruction is, in a large part, due to the quasi-linear nature of the Einstein equations. This gives a significant complication to our method when compared to purely semi-linear estimates; see \cite{Kazdan, Taubes}. To control these
terms, we adapt an idea from Alexakis \cite{Alexakis, AIK} to directly
control the connection coefficients. As explained in Section
\ref{PDEODESystemSection} below, we consider a radially transported orthonormal frame 
over a geodesic normal co\"ordinate system. This way we obtain a system 
of first order ordinary differential equations on the co\"ordinate 
coefficients of the frame vectors, and on the connection coefficients. 
Insofar as ODEs are concerned, the analogue for Carleman estimates is 
almost trivial. 

In summary, the unique continuation problem for the Einstein-scalar-field
system is reduced to a mixed partial-differential- and
ordinary-differential-equation system for the \emph{differences}
between two solutions (e.g.
$\delta\phi = \phi - \bar{\phi}, \delta \nabla\phi = \partial\phi -
\partial\bar{\phi}, \delta R = R - \bar{R}, \ldots$). By imposing 
the condition that the data for this system vanishes to infinite 
order at a point, we can apply Carleman estimates to the system and
obtain local uniqueness of the solution. 

The authors would like to note that Biquard \cite{Biquard} also
obtained a proof of strong uniqueness for the case of the pure 
Einstein metric using Carleman inequalities. Our approaches are,
however, fundamentally different as Biquard uses a system of elliptic
equations on the metric in local co\"ordinates, and on the second
fundamental form for a foliation of the ball by geodesic distance to the
origin, while we consider the tensorial elliptic system on the
curvature level. 

\section{Carleman Estimates and Uniqueness}

In this section we study the strong unique continuation for a 
coupled PDE-ODE system as a model problem. As explained in the
introduction, the motivation for the model problem lies in the need to
estimate connection coefficients \`a la \cite{Alexakis} and
\cite{AIK}. The reduction to the model problem will be explained in
more detail in Section \ref{PDEODESystemSection}. To show the strong
unique continuation property we shall need two
Carleman type estimates. The first estimate, due to Sogge \cite{Sogge}, is
an $L^p$-Carleman inequality adapted to show strong
unique continuation for elliptic systems with rough coefficients. It
will be used to control the PDE part of the system. A second, more
trivial, estimate used to control the ODE part of the system will be
shown below. 

We now work on $\mathbb{R}^n$, $n>2$, with the standard metric. Let
$\rh = r(1-r^\de)$ where $0 < \de < \frac{2}{n}$ is a small
constant. We fix $p = \frac{2n}{n+2}$ and $q = \frac{2n}{n-2}$.
The following easy fact will be crucial in the proof of strong unique
continuation:
\[ p < 2 < q\] 
since it allows one to use H\"older's inequality to get a small
constant by shrinking the neighborhood around the origin.
For any $\la$ such that 
\begin{equation}\label{lambdacondition}
|\la -(k +\frac{n-2}{2})| \geq \frac{1}{2}
\end{equation}
holds for every $k=1,2,...$, we have the following Carleman type 
inequality:
\begin{lemma}[Sogge \cite{Sogge}]\label{SoggeCarleman}
Assuming $u$ is a smooth function compactly supported in $B_{R}
\setminus\{0\}$ for some sufficiently small $R$, then there 
is a constant $C$ independent of 
$\lambda$, such that for all $\lambda$ sufficiently large and
satisfying \eqref{lambdacondition}
\begin{equation}\label{Carleman1}
\| \rh ^{-\lambda} u\|_{L^q(B_R)} + \lambda ^{\frac{1}{n}}\|\rh
^{-\lambda}r^{-1 + \frac{\de}{2}}\nabla u\|_{L^q(B_R)} \leq C
\|\rh^{-\la} \triangle_g u\|_{L^p(B_R)}
\end{equation}
\end{lemma}

\begin{remark} Observe that the previous estimate holds for a
sequence of $\lambda$'s unbounded in $\mathbb{R}_+$. So we can 
let $\lambda$ tend to infinity, which gives us the 
``small constant'' necessary for the strong unique continuation.
\end{remark}

In the following, $Y$ denotes the smooth radial vector field 
$r\partial_r$.

\begin{lemma}\label{BabyCarleman}
For any $\delta$
fixed sufficiently small (depending on the dimension $n$), there 
exists $R_0 = R_0(\delta)$ and $C = C(n,R_0,\delta)$, such that for
any $\lambda > n$ and $R<R_0$, if $u$ is a smooth function compactly
supported in $B_R$ which vanishes to infinite order at the origin, 
\begin{equation}\label{Carleman2}
\| \rh ^{-\lambda} u\|_{L^2(B_R)}  \leq
\frac{C}{\lambda} \|\rh^{-\la} Y(u)\|_{L^2(B_R)}
\end{equation}
\end{lemma}

\begin{proof} We use the standard measure $dx_1 \wedge \cdots \wedge
dx_n$. The measure $d\sigma$ is the standard measure on unit
sphere $\S$.
\begin{align*}
\| \rh ^{-\lambda} u\|^2_{L^2(B_R)} &= c_n \int_0^R \int_{\S}
r^{-2\la}(1-r^\de)^{-2\la} u^2 r^{n-1}dr d\sigma\\
&= \frac{c_n}{-2 \la +n} \int_0^R \int_{\S} (1-r^\de)^{-2\la} u^2
d(r^{-2\la + n}) d\sigma\ \\
&=\frac{2c_n}{2 \la -n} \int_0^R \int_{\S} (1-r^\de)^{-2\la} u
\cdot \p_r u \cdot r^{-2\la +n}dr d\sigma \\
&\quad + \frac{2c_n \la\de}{2 \la -n} \int_0^R \int_{\S}
\frac{r^\de}{1-r^\de} r^{-2\la} (1-r^\de)^{-2\la} u^2 \cdot
r^{n-1}dr d\sigma
\end{align*}
We can assume $\de$ to be fixed $< c_n / 2^{10}$, and $R_0$ fixed 
so that the factor $\frac{r^\de}{1-r^\de} \leq 1$ for all $r<R<R_0$. 
Then since $\lambda > n$, the second term on the right hand side 
can be absorbed to the left hand side, and we have
\begin{align*}
\| \rh ^{-\lambda} u\|^2_{L^2(B_R)} &\leq \frac{C}{\la} \int_0^R
\int_{\S} (1-r^\de)^{-2\la} u \cdot \p_r u \cdot r^{-2\la +n}dr
d\sigma\\
&=\frac{C}{\la} \int_0^R \int_{\S} \rh^{-\la} u \cdot \rh ^ {-\la}
r\p_r u \cdot r^{n-1}dr d\sigma\\
&=\frac{C}{\la} \int_0^R \int_{\S} \rh^{-\la} u \cdot \rh ^ {-\la}
Y(u) \cdot r^{n-1}dr d\sigma\\
&\leq \frac{C}{\lambda}\| \rh ^{-\lambda} u\|_{L^2(B_R)} \cdot \|
\rh ^{-\lambda} Y(u)\|_{L^2(B_R)}
\end{align*}
This proves the lemma.
\end{proof}

The previous two estimates allows one to use the standard Carleman  
techinques to show the strong unique continuation for the following
model problem:

\begin{proposition}\label{modelproblem}
Given a smooth solution $(u,v)$ of the following PDE-ODE system
\begin{equation}\label{model}
\left\{ \begin{array}{rl}
    \triangle_g u &= u + \nab u + v\\
    Y(v) &= v + u + \nab u
        \end{array}\right.
\end{equation}
 If $(u,v)$ vanishes to infinite order around a point $p \in M$, then
 there is a small neighborhood $U$ of $p$ such that
 \begin{equation}\label{zero}
 u=v=0 \qquad \text{on} \quad U.
 \end{equation}
\end{proposition}

\begin{proof} Let $R$ be sufficiently small so that Lemmas
\ref{SoggeCarleman} and \ref{BabyCarleman} both apply.  
Let $\phi$ be a cut-off function supported in the
unit ball $B_1$ such that $\phi|_{B_R} = 1$. Let
\begin{equation}\label{cutoff}
    \chi_k(x) = 1 - \phi(2^k x)~.
\end{equation}

We then define
\begin{equation}\label{ukvk}
\left\{ \begin{array}{rl}
    u_k (x) &= \chi_k(x) \cdot \phi(x) \cdot u(x)\\
    v_\phi(x) &= \phi(x) \cdot v(x)
        \end{array}\right.
\end{equation}

We can apply the Carleman estimate \eqref{SoggeCarleman} to 
$u_k$ in $B_1$ which gives:

\begin{align}\label{ucarleman}
\| \rh ^{-\lambda} u_k\|_{L^q(B_R)} &+ \lambda ^{\frac{1}{n}}\|\rh
^{-\lambda}r^{-1 + \frac{\de}{2}}\nabla u_k\|_{L^q(B_R)} \\ 
\nonumber & \leq
C (\|\rh^{-\la} \triangle_g u_k\|_{L^p(B_R)}+ \|\rh^{-\la}
\triangle_g (\phi \cdot u)\|_{L^p(^c B_R)})
\end{align}
where ${}^c B_R$ denotes the complement of $B_R$ in $B_1$. 
We use Leibnitz rule and the equation \eqref{model} to get
\begin{align*}
\|\rh^{-\la} \triangle_g u_k &\|_{L^p(B_R)} = \|\rh^{-\la} (\chi_k
\triangle_g u + 2 \nabla_g \chi_k \cdot \nabla_g u+
u\triangle_g \chi_k )\|_{L^p(B_R)}\\
&\leq \|\rh^{-\la} \chi_k \triangle_g u
\|_{L^p(B_R)} + R^{-\lambda}2^{(\lambda + 2)
k}\|\phi\|_{C^2(B_1)}\|u\|_{W^{1,p}(B_{2^{-k}})}
\end{align*}
where we used that $\nabla\chi_k$ is supported on $B_{2^{-k}}\setminus
B_{2^{-k}R}$. Next we note that if $u$ vanishes to infinite order at
the origin, then $\|u\|_{W^{1,p}(B_\rho)} = O(\rho^N)$ for any $N > 0$
as $\rho\searrow 0$, so for any fixed $R,\lambda$, as $k\nearrow\infty$
the second term above tends to 0. We denote this behaviour by
$o_k(1)$. Continuing 
\begin{align*}
\|\rh^{-\la} \triangle_g u_k &\|_{L^p(B_R)} \leq \|\rh^{-\la} \chi_k u \|_{L^p(B_R)} + \|\rh^{-\la} \chi_k
\nabla u \|_{L^p(B_R)} + \|\rh^{-\la} \chi_k v \|_{L^p(B_R)} + o_k(1)\\
&\stackrel{H\ddot{o}lder}{\leq} R^2\|\rh^{-\la} \chi_k u
\|_{L^q(B_R)} + R^2 \|\rh^{-\la} \chi_k \nabla u \|_{L^q(B_R)} +
R\|\rh^{-\la} \chi_k v \|_{L^2(B_R)}+o_k(1)
\end{align*}
Plug this in \eqref{ucarleman}, and for any fixed $\lambda, R$, let 
$k\nearrow\infty$, we have
\begin{align}\label{uestimate}
\| \rh ^{-\lambda} u_k\|_{L^q(B_R)} &+ \lambda ^{\frac{1}{n}}\|\rh
^{-\lambda}r^{-1 + \frac{\de}{2}}\nabla u_k\|_{L^q(B_R)} \\ 
\nonumber & \leq
C (R^2\|\rh^{-\la} \chi_k u
\|_{L^q(B_R)} + R^2 \|\rh^{-\la} \chi_k \nabla u \|_{L^q(B_R)} \\
\nonumber & \qquad + R\|\rh^{-\la} \chi_k v \|_{L^2(B_R)}+ \|\rh^{-\la} \triangle_g
(\phi \cdot u)\|_{L^p(^c B_R)})
\end{align}

Now to control $v$, we need to use the Carleman
estimate \eqref{Carleman2} and the equation \eqref{model} to
estimate $v_\phi(x)$:
\begin{align*}
\lambda \| \rh ^{-\lambda} v\|_{L^2(B_R)} & \leq
C(\|\rh^{-\la} Y(v)\|_{L^2(B_R)} + \|\rh^{-\la} Y(\phi
v)\|_{L^2(^cB_R)})\\
&\leq C(\|\rh^{-\la} v\|_{L^2(B_R)} + \|\rh^{-\la} u\|_{L^2(B_R)} \\
& \qquad +
\|\rh^{-\la} \nabla u\|_{L^2(B_R)}+ \|\rh^{-\la} Y(\phi
v)\|_{L^2(^cB_R)})
\end{align*}
Applying H\"older's inequality to the second and third terms, 
\begin{align}\label{vestimate}
\lambda \| \rh ^{-\lambda} v\|_{L^2(B_R)} & \leq C( \|\rh^{-\la}
v\|_{L^2(B_R)} \\
\nonumber & \quad + 
R\|\rh^{-\la} u\|_{L^q(B_R)}+ R\|\rh^{-\la} \nabla u\|_{L^q(B_R)}+
\|\rh^{-\la} Y(\phi v)\|_{L^2(^cB_R)})
\end{align}
Now, puttin \eqref{uestimate} and \eqref{vestimate} together, we can
assume that $\lambda$ is initially chosen to be large enough, and $R,
\delta$ are chosen initially small enough such that the first three
terms on the right hand side of \eqref{uestimate} and the first three
terms on the right hand side of \eqref{vestimate} can all be absorbed
to the left hand side. And so we have
\begin{align*}
\| \rh ^{-\lambda} u\|_{L^q(B_R)} + &\lambda ^{\frac{1}{n}}\|\rh
^{-\lambda}r^{-1 + \frac{\de}{2}}\nabla u\|_{L^q(B_R)} +
\lambda \| \rh ^{-\lambda} v\|_{L^2(B_R)} \\&\leq C
(\|\rh^{-\la} \triangle_g (\phi u)\|_{L^p(^c B_R)}+
\|\rh^{-\la} Y(\phi v)\|_{L^2(^cB_R)})\\
&\leq C \Rh^{-\la}(\|\triangle_g (\phi u)\|_{L^p(^c B_R)}+
\|Y(\phi v)\|_{L^2(^cB_R)})\\
&\leq C' \Rh^{-\la}
\end{align*}
In particular, it shows
\[ \|(\frac{\Rh}{\rh})^\lambda u \|_{L^q(B_R)} \leq C'~,\quad
\|(\frac{\Rh}{\rh})^\lambda v \|_{L^q(B_R)} \leq \frac{C'}{\lambda} \]
Now let $\lambda \to \infty$ and notice that $\Rh/\rh > 1$, so by Lebesgue 
dominant convergence theorem, $u=v=0$ in $B_R$.
\end{proof}

\begin{remark}\label{technicalremark}
Observe that the same argument also applies if instead of the system
\eqref{model}, we have the differential inequality
\begin{equation}\label{model2}
\left\{ \begin{array}{rl}
    |\triangle_g u| & \leq C( |u| + |\nab u| + |v|)\\
    |Y(v)| &\leq C(|v| + |u| + |\nab u|)
        \end{array}\right.
\end{equation}
instead. This form is more convenient in applications.
\end{remark}

\section{PDE-ODE System}\label{PDEODESystemSection}

Fix a point $ p_0 \in M$ and an orthonormal basis
$\{e_i(p_0)\}$ on $p_0$. By parallel transporting this frame
along geodesics emanating from $p_0$, we get an orthonormal frame
$\{e_i\}$ near $p_0$. Let the smooth vector field
 $$Y= r \p_r $$ 
be defined on a normal neighborhood near $p_0$, where $r$ is the
geodesic distance to $p_0$. By the
construction of $\{e_i\}$, we have
\begin{equation}\label{nabY}
 \nab_Y e_i = 0~, \quad \nab_Y Y = Y~.
\end{equation}
The connection coefficients relative to this frame are expressed as
\[ \G_{i j}^{k} = g(\nab_{e_i} e_j, e_k)~, \qquad \G_{i Y}^{j} =
g(\nab_{e_i} Y, e_j)~. \]
We fix a co\"ordinate system $(x_\a)$, then
\[Y = Y^\a \p_\a~, \quad e_i = e_i^\a \p_\a~.\]
Since
\begin{align*}
 \G_{iY}^k e_k^\a \p_\a&=\G_{iY}^k e_k = \nab_{e_i} Y = [e_i, Y] \\
&= (\p_\b Y_\a) e_i^\b \p_\a - (Y e_i^\a)\p_\a
\end{align*}
we have
\begin{equation}\label{Yeia}
         Y e_i^\a = (\p_\b Y^\a) e_i^\b - \G_{iY}^{k}e_k^\a~.
\end{equation}
We now compute the derivatives of the connection coefficients along
$Y$
\begin{align*}
Y\G_{i j}^{k} &= Y g(\nab_{e_i} e_j, e_k) = g(\nab_Y \nab_{e_i} e_j,
e_k)\\
&=  g(R(Y,e_i)e_j,e_k)+g(\nab_{e_i} \nab_Y e_j, e_k)+g(\nab_{\nab_Y
{e_i}-\nab_{e_i} Y}e_j, e_k)\\
&=R_{Yijk}-\G_{iY}^{p}\G_{pj}^{k}\\
Y\G_{i Y}^{k} &= Y g(\nab_{e_i} Y, e_k) = g(\nab_Y \nab_{e_i} Y,
e_k)\\
&= g(R(Y,e_i)Y,e_k)+g(\nab_{e_i} \nab_Y Y, e_k)+g(\nab_{\nab_Y
{e_i}-\nab_{e_i} Y}Y, e_k)\\
&= R_{YiYk}+g(\nab_{e_i} Y, e_k)-g(\nab_{\nab_{e_i} Y}Y, e_k)\\
&= R_{YiYk}+\G_{i Y}^{k}- \G_{i Y}^{p}\G_{p Y}^{k}
\end{align*}
together with (\ref{Yeia})
\begin{equation}\label{gamma1}
\left\{ \begin{array}{rl}
    Y (e_i^\a) &= (\p_\b Y^\a) e_i^\b - \G_{iY}^{k}   e_k^\a\\
         Y(\G_{i j}^{k}) &= R_{Yijk}-\G_{iY}^{p}   \G_{pj}^{k}\\
       Y(\G_{i Y}^{j}) &= R_{YiYj}+\G_{i Y}^{j}- \G_{i Y}^{p}   \G_{p
Y}^{j}
        \end{array}\right.\end{equation}
Differentiating this system by co\"ordinate derivatives, we have
(by convention, unless bracketed, the derivatives only acts on the 
multiplicand immediately following it)
\begin{equation}\label{dgamma1}
\left\{ \begin{array}{rl}
    Y (\p_\b e_i^\a) &= -\p_\b Y^\g   \p_\g e_i^\a + \p_\b\p_\g Y^\a
e_i^\g + \p_\g Y^\a   \p_\b e_i^\g - \p_\b \G_{iY}^{k}   e_k^\a
-\G_{iY}^{k}  \p_\b e_k^\a\\
         Y (\p_\a \G_{i j}^{k}) &= -\p_\a Y^\b   \p_\b \G_{i j}^{k} +
\p_\a R_{Yijk}- \p_\a\G_{i Y}^{p}   \G_{p j}^{k} - \G_{iY}^{p}
\p_\a\G_{p j}^{k}\\
       Y (\p_\a \G_{i Y}^{j}) &= -\p_\a Y^\b   \p_\b \G_{i Y}^{j} +
\p_\a R_{YiYj}+ \p_\a \G_{i Y}^{j} - \p_\a\G_{i Y}^{p}   \G_{p Y}^{j}
- \G_{iY}^{p}   \p_\a\G_{p Y}^{j}
        \end{array}\right.\end{equation}
For the curvature terms in \eqref{gamma1} and \eqref{dgamma1}, they
involve the vector field $Y$. We still need to untangle it into the
frame components. Let $\p_\a = e^i_\a   e_i$, so $e^i_\a   e^\b_i =
\delta_\a^\b$. So we have
\begin{align*}
 Y(e^i_\a)&=-e^i_\b e_\a^j Y(e_j^\b) =-e^i_\b e_\a^j(\p_\g Y^\b e_j^\g
- \G_{jY}^{k}   e_k^\b) =-\p_\a Y^\b   e^i_\b  +e_\a^j   \G_{jY}^{i}\\
 Y(\p_\b e^i_\a)&=[Y,\p_\b] e^i_\a+ \p_\b (Y(e^i_\a))\\
&=-\p_\b Y^\g   \p_\g e^i_\a -\p_\a \p_\b Y ^\g   e_\g^i -\p_\a
Y^\g   \p_\b e_\g^i + \p_\b e_\a^j    \G_{j Y}^{i} + e_\a^j
\p_\b \G_{j Y}^{i}
\end{align*}
Now we can untangle the curvature terms for \eqref{gamma1}
\begin{align*}
Y(\G_{i j}^{k}) &= R_{Yijk}-\G_{iY}^{p}   \G_{pj}^{k} =Y^\g e_\g^l
R_{lijk} -\G_{iY}^{p}   \G_{pj}^{k}\\
Y (\p_\a \G_{i j}^{k}) &= -\p_\a Y^\b   \p_\b \G_{i j}^{k} + \p_\a
R_{Yijk}- \p_\a\G_{i Y}^{p}   \G_{p j}^{k} - \G_{iY}^{p}   \p_\a\G_{p
j}^{k}\\
&=-\p_\a Y^\b   \p_\b \G_{i j}^{k} +\p_\a Y^\g e_\g^l
R_{lijk}+Y^\g \p_\a e_\g^l R_{lijk} \\&\quad + Y^\g e_\g^l \p_\a
R_{lijk}-\p_\a\G_{i Y}^{p}   \G_{p j}^{k} - \G_{iY}^{p}
\p_\a\G_{p j}^{k}\\
Y(\G_{i Y}^{j}) &= R_{YiYj}+\G_{i Y}^{j}- \G_{i Y}^{p}   \G_{p
Y}^{j}\\
&=Y^\a  Y^\b   e_\a^k   e_\b^l   R_{kilj}+\G_{i Y}^{j}- \G_{i
Y}^{p}   \G_{p Y}^{j}\\
Y (\p_\a \G_{i Y}^{j}) &= -\p_\a Y^\b   \p_\b \G_{i Y}^{j} + \p_\a
R_{YiYj}+ \p_\a \G_{i Y}^{j} - \p_\a\G_{i Y}^{p}   \G_{p Y}^{j} -
\G_{iY}^{p}   \p_\a\G_{p Y}^{j}\\
&= -\p_\a Y^\b   \p_\b \G_{i Y}^{j} + \p_\a \G_{i Y}^{j} - \p_\a\G_{i
Y}^{p}   \G_{p Y}^{j} - \G_{iY}^{p}   \p_\a\G_{p Y}^{j} + \p_\a (Y^\g
Y^\b) e_\g^k e_\b^l R_{kilj}\\
& \quad + Y^\g Y^\b
\p_\a e_\g^k e_\b^l R_{kilj} +Y^\g Y^\b e_\g^k \p_\a e_\b^l
R_{kilj} +Y^\g Y^\b e_\g^k e_\b^l \p_\a R_{kilj}
\end{align*}
We can put all these equations together to have the following system
of equations for connection coefficients:
\begin{equation}\label{eqGamma}
\left\{ \begin{array}{rl}
    Y(e_i^\a) &= \p_\b Y^\a e_i^\b - \G_{iY}^{k}   e_k^\a\\
    Y(e^i_\a) &= -\p_\a Y^\b   e^i_\b  +e_\a^j   \G_{jY}^{i}\\
    Y (\p_\b e_i^\a) &= -\p_\b Y^\g   \p_\g e_i^\a + \p_\b\p_\g Y^\a
e_i^\g + \p_\g Y^\a   \p_\b e_i^\g - \p_\b \G_{iY}^{k}   e_k^\a
-\G_{iY}^{k}  \p_\b e_k^\a\\
    Y(\p_\b e^i_\a)&=-\p_\b Y^\g   \p_\g e^i_\a -\p_\a \p_\b Y ^\g
e_\g^i -\p_\a Y^\g   \p_\b e_\g^i + \p_\b e_\a^j    \G_{j Y}^{i} +
e_\a^j   \p_\b \G_{j Y}^{i}\\
        Y(\G_{i j}^{k}) &= Y^\g e_\g^l R_{lijk} -\G_{iY}^{p}
\G_{pj}^{k}\\
        Y(\G_{i Y}^{j}) &= Y^\a  Y^\b   e_\a^k   e_\b^l
R_{kilj}+\G_{i Y}^{j}- \G_{i Y}^{p}   \G_{p Y}^{j}\\
    Y (\p_\a \G_{i j}^{k}) &= -\p_\a Y^\b   \p_\b \G_{i j}^{k} +\p_\a
Y^\g e_\g^l R_{lijk}+Y^\g \p_\a e_\g^l R_{lijk} \\
&\quad + Y^\g e_\g^l \p_\a R_{lijk}-\p_\a\G_{i Y}^{p}   \G_{p j}^{k} -
\G_{iY}^{p}   \p_\a\G_{p j}^{k}\\
    Y (\p_\a \G_{i Y}^{j}) &= -\p_\a Y^\b   \p_\b \G_{i Y}^{j} + \p_\a
\G_{i Y}^{j} - \p_\a\G_{i Y}^{p}   \G_{p Y}^{j} - \G_{iY}^{p}
\p_\a\G_{p Y}^{j}\\
& \quad + \p_\a (Y^\g Y^\b) e_\g^k e_\b^l R_{kilj}  + Y^\g Y^\b \p_\a
e_\g^k e_\b^l R_{kilj} \\
& \quad +Y^\g Y^\b e_\g^k \p_\a e_\b^l R_{kilj} +Y^\g Y^\b e_\g^k
e_\b^l \p_\a R_{kilj}
        \end{array}\right.
\end{equation}

Let $v$ to denote the vector value function which is the collection of
$e_i^\a$, $\p_\b
e_i^\a$, $\G_{i j}^{k}$, $\G_{i Y}^{j}$, $\p_\a \G_{i j}^{k}$ and
$\p_\a \G_{i Y}^{j}$, more precisely, we define
\begin{equation}
v = (e_i^\a, \p_\b e_i^\a, \G_{i j}^{k}, \G_{i Y}^{j}, \p_\a \G_{i
j}^{k}, \G_{i Y}^{k}).
\end{equation}
Schematically, the system \eqref{eqGamma} can be written as
\begin{equation}\label{eq1}
 Y(v) = A*v + A*v*v + A*v*R + A*v* \p R +A*v*v*R + A*v*v* \p R
\end{equation}
where $A$ denotes some algebraic expression involving $Y^\a, \p_\beta
Y^\a$ and $\p R$ denotes $\p_\a R_{kilj}$.

An Einstein-scalar field on a Riemannian manifold is a triplet
$(M,g,\phi)$ satisfying the following system of equations
\begin{equation}\label{einsteinn}
\left\{ \begin{array}{rl}
    R_{ij}-\frac{1}{2} g_{ij} R &=\nab_i \phi \nab_j \phi -\frac{1}{2}
g_{ij}(|\nab \phi|^2+2V(\phi)) - \lambda g_{ij}\\
    \triangle_g \phi &= V'(\phi)
        \end{array}\right.
\end{equation}
where $R_{ij}$ the Ricci curvature of the metric $g_{ij}$, $R$ is the
corresponding scalar curvature and $\phi$ is the coupled scalar field.
$\lambda$ is a real number that denotes the cosmological constant, and
the function $V:\mathbb{R}\to\mathbb{R}$ is the potential for the
scalar field.
We use $\nab$ to denote the Levi-Civita connection for $g$ and
$\triangle$ for the corresponding laplacian.
By taking the trace of the first equation in \eqref{einsteinn}, it can 
be simplified as follows:
\begin{equation}\label{einstein}
\left\{ \begin{array}{rl}
    R_{ij}&=\nab_i \phi \nab_j \phi + g_{ij}V(\phi) + g_{ij} \lambda\\
    \triangle_g \phi &= V'(\phi)
        \end{array}\right.
\end{equation}

To simplify the formulae we will use $*$
notation. The expression $A*B$ is a linear combination of tensors,
each formed by starting with $A \otimes B$, using the metric to
take any number of contractions. So the algorithm to get $A*B$ is
independent of the choices of tensors $A$ and $B$ of respective
types.

Recall the second Bianchi identity
\begin{equation}\label{bianchi}
\nab_i R_{jklm} +\nab_j R_{kilm} +\nab_k R_{ijlm}=0
\end{equation}
We trace the $i$ and $l$ indices, then we have the once contracted
Bianchi identity
\begin{align*}
\nab^i R_{jkim} &= -\nab_j R_{k}{}^i{}_{im} -\nab_k R^{i}{}_{jim}
=-\nab_j R_{km} + \nab_k R_{jm}\\
&=(\nab^2 \phi* \nab \phi)_{jkm} + (V''(\phi) \nab \phi)_{jkm}
\end{align*}
We now apply divergence operator on the second Bianchi identity
(\ref{bianchi})
and use the previous contracted version,  we have
\begin{align*}
 \nab^i \nab_i R_{jklm} &= - \nab^i \nab_j R_{kilm} - \nab^i \nab_k
R_{ijlm} \\
&= -[\nab^i, \nab_j]R_{kilm}-[\nab^i, \nab_k] R_{ijlm} - \nab_j \nab^i
R_{kilm} -\nab_k  \nab^i  R_{ijlm} \\
&= (R*R)_{jklm} +(\nab^2 \phi * \nab^2 \phi)_{jklm} +(\nab^3 \phi *
\nab \phi)_{jklm}\\
&\quad +  (V''(\phi)\nab^2 \phi)_{jklm} + (V'''(\phi)*\nab \phi*\nab
\phi)_{jklm}
\end{align*}
So we have the following equations for the full curvature tensor:
\begin{align}
\triangle_g R_{ijkl}& = (R*R)_{ijkl} +(\nab^2 \phi * \nab^2
\phi)_{ijkl} +(\nab^3 \phi * \nab \phi)_{ijkl}\\
&\quad +  (V''(\phi)\nab^2 \phi)_{ijkl} + (V'''(\phi)*\nab \phi*\nab
\phi)_{ijkl}
\end{align}
Now we deal with the scalar field part of \eqref{einstein}. We apply the 
covariant derivatives on the equations, then commute derivatives, it gives
\begin{align*}
V''(\phi) \nab_j \phi &= \nab_j \triangle_g \phi=\nab_j \nab^i \nab_i \phi\\
&=\nab^i \nab_j  \nab_i \phi - R_j{}^i{}_{ik}\nab^k\phi=\triangle_g \nab_j \phi - R_{jk}\nab^k\phi
\end{align*}
So we have
\begin{equation}\label{nabphi}
 \triangle_g \nab \phi = V''(\phi)*\nab \phi + R*\nab \phi
\end{equation}
We repeat this to get
\begin{equation}\label{nabnabphi}
 \triangle_g (\nab^2 \phi) = V''(\phi)*\nab^2 \phi +V'''(\phi)*\nab \phi *\nab \phi+\nab R*\nab \phi + R*\nab^2 \phi
\end{equation}
In summary, we have
\begin{equation}\label{field}
\left\{ \begin{array}{rl}
\triangle_g R_{ijkl}& = (R*R)_{ijkl} +(\nab^2 \phi * \nab^2
\phi)_{ijkl} +(\nab^3 \phi * \nab \phi)_{ijkl}\\
&\quad +  (V''(\phi)\nab^2 \phi)_{ijkl} + (V'''(\phi)*\nab \phi*\nab
\phi)_{ijkl}\\   
\triangle_g \nab \phi &= V''(\phi)*\nab \phi + R*\nab \phi\\
\triangle_g (\nab^2 \phi) &= V''(\phi)*\nab^2 \phi +V'''(\phi)*\nab \phi *\nab \phi+\nab R*\nab \phi + R*\nab^2 \phi
        \end{array}\right.
\end{equation}
The above is a tensor equation written in abstract index notation. To
convert the system to a scalar system of equations, we calculate the
frame components of both the left and right hand sides. In particular,
we let $u$ be vector valued function
\begin{equation}
u = (R_{ijkl}, \phi, \nabla_i\phi, \nabla^2_{ij}\phi)
\end{equation}
comprising the scalar $\phi$ and the frame components of 
$R$, $\nabla\phi$, and
$\nabla^2\phi$. The scalar equations for $u$ will then involve
connection coefficients where terms are expressed as derivatives of
the tensors $R$, $\nabla\phi$, and $\nabla^2\phi$. Then it is clear 
that, since the left hand side of \eqref{field} has at most two 
derivatives of $R$, $\nabla\phi$, and $\nabla^2\phi$, while the right 
hand side has at most one derivative, the equation for $u$ can be
written down schematically (noting that the equation for $\phi$ is
given in \eqref{einstein}) as
\begin{equation}\label{eq3}
 \triangle_g u = v*v*u + v*u + \p u * u + v*u*u + B*u + B*u*u 
\end{equation}
noting that while terms involving one derivative of the connection
coefficients appear, no terms involve two derivatives of the
connection coefficients, and so no derivatives of $v$ are involved;
$B$ here denotes coefficients involving $V'(\phi)$, $V''(\phi)$ and 
$V'''(\phi)$
(henceforth we require the potential $V:\mathbb{R}\to\mathbb{R}$ to
have regularity $C^{3,1}$). 

In view of \eqref{eq1}, we have the following closed PDE-ODE system
\begin{equation}\label{main}
 \left\{ \begin{array}{rl}
\triangle_g u &= B*(u + u*u) + v*u*(1 + u + v) + u*\p u\\
 Y(v) &= A*v*(1 + v + u + \p u + v*u + v*\p u)
        \end{array}\right.
\end{equation}
We now turn to the strong unique continuation for the Einstein-scalar field: 
\begin{theorem}\label{maintheorem}
 Let $U$ and $\bar{U}$ be smooth manifolds with dimension $n\geq 3$. Let $(g,\phi)$ and $(\gb,\phib)$ be solutions to the Einstein-scalar equation \eqref{einstein} on $U$ and $\bar{U}$ respectively. We assume that there exist points $p\in U$ and $\bar{p} \in \bar{U}$ such that the metric tensors, the covariant derivatives of arbitrary order of the matter fields, and the covariant derivatives of arbitrary order of the Riemann curvature tensors for the two solutions agree. In other words, we assume that there exist a vector-space isomorphism $\Psi: T_pU \to T_{\bar{p}}\bar{U}$ such that, 
\begin{align*}
g(X_{(1)}, X_{(2)}) &= \gb(\Psi_*X_{(1)},\Psi_*X_{(2)}) \\
\nabla^{(k)}_{X_{(1)}\ldots X_{(k)}}\phi &=  \bar{\nabla}^{(k)}_{\Psi_*X_{(1)}\ldots \Psi_*X_{(k)}}\phi \\
\nabla^{(k)}_{X_{(1)}\ldots X_{(k)}}R(X_{(k+1)}, \ldots, X_{(k+4)}) &= \bar{\nabla}^{(k)}_{\Psi_*X_{(1)}\ldots \Psi_*X_{(k)}}\Rb(\Psi_*X_{(k+1)}, \ldots, \Psi_*X_{(k+4)})\end{align*}
for all $X_{(j)}\in T_pU$. Then there exists two neighborhoods $p\in V \subset U$, $\bar{p} \in \bar{V} \subset \bar{U}$ and a diffeomorphism $f:V \to \bar{V}$, such that $(g,\phi)$ and $(f^*\gb,f^*\phib)$ agree on $V$.
\end{theorem}
\begin{proof} Without loss of generality, we can assume that $U = \bar{U}$, $p = \bar{p}$, and $\Psi = Id$. (We can always extend $\Psi$ to a smooth diffeomorphism of a (possibly smaller) neighborhood of $p$ to a neighborhood of $\bar{p}$, and replace $\gb$ and $\phib$ by their pull-backs.) We begin by considering the construction of the diffeomorphism $f$. Take normal co\"ordinates $\{x^\a\}$ and $\{y^\a\}$ for the metrics $g$ and $\gb$ respectively, such that the vectors $\partial_{x^\a}  |_p = \partial_{y^\a}$. Using the standard Taylor expansion for the metric in normal co\"ordinate system
\[ g_{ij} = \delta_{ij} - \frac{1}{3}R_{i\alpha\beta j}x^\alpha x^\beta + O(|x|^3) \]
we see that the ``infinite order agreement'' of the Riemann curvature tensor assumed in the statement of the theorem implies that all co\"ordinate derivatives
\[ \partial^{(k)}_{x^\alpha \ldots x^\kappa} g_{\mu\nu} (p) = \partial^{(k)}_{y^\alpha \ldots y^\kappa} \bar{g}_{\mu\nu}(p) \]
for $g$ and $\gb$ in their respective normal co\"ordinates agree at $p$. This defines our diffeomorphism $f$. In what follows, all of the barred quantities (e.g. $\gb, \Rb, \phib$) are implicitly pulled-back via the diffeomorphism $f$, and so we work on a fixed, common co\"ordinate system. The crucial advantage of this construction is that the co\"ordinate expressions for the radial vector field $Y = r\p_r$ is the same.

For  $(g, \phi)$, we have the PDE-ODE system \eqref{main}; for $(\gb,\phib)$, the same construction applies. We simply replace each quantity by the same letter with a bar above it.

\begin{equation}\label{mainb}
 \left\{ \begin{array}{rl}
\triangle_{\gb} \ub &= \bar{B}*(\ub + \ub*\ub) + \vb * \ub * (1 + \ub + \vb)
+ \ub*\partial\ub\\
 Y(\vb) &= A*\vb*(1+ \vb + \ub + \partial\ub + \vb*\ub + \vb*\p \ub)
        \end{array}\right.
\end{equation}

Let $\delta u =u -\ub$ and $\delta v = v -\vb$. We take the difference
of \eqref{main} and \eqref{mainb}, as we commented in the
introduction, since we have already coupled the Christoffel symbols in
the system, every term can be written as the product of a fixed
quatity with either $\delta u$ or $\delta v$. In fact, since $Y$ is
the same for both metrics, the expression for $A$ is the same for the
two systems and will not cause any trouble; for the differences in $B$, 
since we assume $V$ is $C^{3,1}$, it can be bounded by $C|\phi
-\phib|$ which is $C \delta u$. In summary, we have
\begin{equation*}
\left\{ \begin{array}{rl}
    |\triangle_g \delta u| & \leq C( |\delta u| + |\nab \delta u| + |\delta v|)\\
    |Y(\delta v)| &\leq C(|\delta v| + |\delta u| + |\nab \delta u|)
        \end{array}\right.
\end{equation*}
Now the Remark \ref{technicalremark} applies which shows $\delta u$
and $\delta v$ both vanish identically in a small neighborhood of $P$. 
The theorem follows.
\end{proof}

Combining the strong unique continuation properties proved in the above theorem, and using techniques from the proof of the monodromy theorem (see, e.g. Theorem 3 in \cite{Myers}), we easily obtain the following corollary:
\begin{corollary}
 Let $U$ and $\bar{U}$ be two simply connected smooth manifolds. Let $(g,\phi)$ and $(\gb,\phib)$ be solutions to the Einstein-scalar equations on $U$ and $\bar{U}$, respectively, and assume that $(U,g)$ and $(\bar{U},\gb)$ are complete. Assuming that there exist points $p\in U$ and $\bar{p}\in \bar{U}$ such that the solutions $(g,\phi)(p)$ and $(\gb,\phib)(\bar{p})$ agree to infinite order (see the above theorem for exact conditions), then there exists a global diffeomorphism $f: U\to\bar{U}$ such that $(g,\phi) = (f^*g, f^*\phi)$ on $U$. 
\end{corollary}
To quickly sketch a proof: using the strong unique continuation property, it is simple to construct isometries in tubular neighborhoods of geodesics originating from $p$ and $\bar{p}$, once we observe that the above theorem guarantees an essentially unique way for this construction (the only possible non-uniqueness comes from a symmetry of the original solutions). From here, by completeness, we guarantee that $(U,g,\phi)$ and $(\bar{U},\gb,\phib)$ are locally identical. The vanishing of the first fundamental group then allows us to patch the local isometries to a well-defined global one in the standard manner.

 \end{document}